\documentclass[12pt]{amsart}
\usepackage{shortcuts}

\title{Additive Problems with Primes from a Thin Bohr Set}

\author{Sarvagya Jain}
\address{Department of Mathematics and Statistics, University of Turku, Turku, Finland}
\email{sarvagyajain.math@gmail.com}

\begin{document}

\begin{abstract}
    For an irrational $\alpha\in \R$, we consider additive problems with the set of primes satisfying $\norm{\alpha p}\leq \frac{1}{p^\tau}$ for some fixed $\tau>0$. In particular, we show that there exist infinitely many non-trivial three-term arithmetic progressions in the set of primes satisfying $\norm{\alpha p}\leq \frac{1}{p^\tau}$ for $\tau\in(0, \tfrac18)$. We also consider a binary Goldbach-type problem.
\end{abstract}

\maketitle
\vspace{-0.5in}
\section{Introduction}
Let $\norm{\cdot}$ denote the distance from a real number to the nearest integer(s).

Let $\alpha$ be a fixed irrational number and $\tau>0$ be fixed.  In this paper, we study additive problems for primes $p$ satisfying
\begin{align}\label{defining condn}
    \norm{\alpha p}\leq p^{-\tau}.
\end{align}

A classical problem in number theory is to measure how closely rationals can approximate an irrational number.  A further refinement is to restrict the denominators to special sets, prime numbers being a particularly natural choice.  In fact, for almost every irrational $\alpha$, there are infinitely many primes $p$ for which~\eqref{defining condn} holds with $\tau=1$ (see, e.g., \cite{duffin-schaeffer}).

When $\alpha$ is fixed, much less is known.  Matomäki~\cite{matomaki}, building on results of Harman~\cite{harman} and of Heath‑Brown and Jia~\cite{heathbrown-jia}, showed that for every fixed $\tau\in(0, \tfrac13)$, there are infinitely many primes $p$ that satisfy~\eqref{defining condn}.  It is widely conjectured that the same holds for all $\tau\in(0, 1)$.  Conversely, Harman~\cite{harman1} proved that for $\tau=1$ there are uncountably many irrationals $\alpha$ for which~\eqref{defining condn} holds for at most finitely many primes $p$.

It is natural to ask whether the primes $p$ satisfying~\eqref{defining condn} exhibit any additional arithmetic structure. For example, Shi~\cite{shi}, improving the work of Matom\"aki \cite{matomaki2}, showed that there are infinitely many Chen primes $p$ satisfying~\eqref{defining condn} for $\tau\in(0, 3/200)$. For other patterns in primes satisfying~\eqref{defining condn}, see also \cite{guo, teravainen, todorova}. These results suggest that the distribution of such primes retains significant structure despite the Diophantine constraint.

Studying three‐term arithmetic progressions within subsets of the primes has long driven advances in additive number theory, giving rise to techniques now central to the field. Van der Corput~\cite{van der corput} first applied the circle method to show that the primes contain infinitely many nontrivial three‐term progressions. Green~\cite{green} later introduced the transference principle, demonstrating that any dense subset of the primes likewise contains infinitely many such progressions. More recently, building on Kelley–Meka’s breakthrough \cite{kelley-meka}, Bloom and Sisask~\cite{bloom-sisask} proved that if
\[
\mathcal{A}\subseteq [1, N]\cap\mathbb{Z}
\]
contains no nontrivial three-term arithmetic progressions, then
\[
\lvert \mathcal{A}\rvert \leq \exp\Bigl(-c (\log N)^{1/9}\Bigr) N.
\]
By the prime number theorem, both Van der Corput’s and Green’s theorems follow immediately.

Nevertheless, by suitably adapting either the circle method or the transference principle, one can still force three-term progressions in extremely sparse yet highly structured subsets of the primes. For example, Du-Pan-Sun~\cite[Theorem~1.3]{yuchen-du-pan} showed that even a very thin set of Piatetski-Shapiro primes contains infinitely many nontrivial three‐term arithmetic progressions.

Analogously, one may ask about three‐term progressions among the primes $p$ satisfying~\eqref{defining condn}, which are known to be sparse. In this work, we adapt the classical circle method to this context and establish the following theorem.
\begin{theorem}\label{theorem 3APs}
Let $\alpha\in\R$ be irrational and $\tau\in\pare{0,\tfrac18}$. Then there exist infinitely many non-trivial triples of primes $(p_1, p_2, p_3)$ satisfying
\[
\|\alpha p_j\|\le \frac{1}{p_j^{\tau}} \quad \text{for } j=1,2,3,
\]
and 
\[
p_1+p_3=2p_2.
\]
\end{theorem}
One might ask whether Green’s Fourier-analytic transference principle can be employed to improve the range of $\tau$ in the theorem above. However, the construction of a majorant satisfying a suitable restriction estimate is obstructed by the Diophantine constraint, as its indicator function exhibits many large Fourier coefficients.

We also consider the binary Goldbach problem: expressing even positive integers as the sum of two primes. The triangle inequality prevents such a decomposition for most even integers when both primes simultaneously satisfy~\eqref{defining condn}. Therefore, we impose the Diophantine condition on only one of them. Our result is the following:

\begin{theorem}\label{theorem binary goldbach}
  Let $\tau\in\pare{0,\tfrac18}$ be fixed. Then for almost all even integers $n$, there exist primes $p$ and $q$ such that
  \[
    p + q = n,
    \quad
    \norm{\alpha p} \le p^{-\tau}.
  \]
\end{theorem}

We remark that the techniques developed here adapt readily to other additive problems not explicitly addressed in this paper.

\subsection{Organisation of the Paper}
We begin in Section~\ref{notation} by introducing the basic notation used throughout the paper. 

In Section~\ref{exponential sum estimates}, we establish a key frequency repulsion estimate. Since the condition~\eqref{defining condn} does not admit an obvious pseudorandomization, we are forced to contend with multiple frequencies simultaneously. To manage this, we exploit Diophantine approximation to isolate one frequency at a time. The arguments are elementary in nature, relying on the pigeonhole principle.

We introduce a smooth weight function in Section~\ref{weight function} and examine its properties. This function serves to smooth the abrupt cutoff imposed by the condition~\eqref{defining condn}, making the analysis easier. 

In Section~\ref{notation setup}, we present a slight modification of the traditional circle method designed to handle multiple frequencies. 
In Section~\ref{Major Arc and Minor Arc Estimates}, we establish major and minor arc estimates that go into our variant of the circle method. 

Finally, in Sections~\ref{proof of theorem 3APs} and~\ref{proof of theorem binary goldbach}, we finish the proof of Theorems~\ref{theorem 3APs} and~\ref{theorem binary goldbach} respectively.

\subsection{Acknowledgements} 
The author would like to thank his advisor, Kaisa Matomäki, for suggesting the problem, for her guidance and suggestions throughout the project, and for a thorough reading of the initial manuscript and remarks that helped improve the exposition. The author was supported by the Academy of Finland, Centre of Excellence (grant numbers 346307 and 364214) and by the University of Turku Graduate School’s Exactus fellowship while working on this project.

\section{Notation}\label{notation}
Throughout, we write $e(\theta)=e^{2\pi i\theta}$, $\|\alpha\|$ for the distance from $\alpha$ to the nearest integer(s), and $\Lambda(n)$ for the von Mangoldt function. We think of $\mbbT$ as the metric space $\pare{\R/\Z, \norm{\cdot}}$. We use the notation $f\ll g$ (or $f=O(g)$) to mean there is a constant $C$ such that $f(x)\le C\,g(x)$ for all sufficiently large $x$, and $f\ll_A g$ (or $f=O_A(g)$) when that constant may depend on the parameter $A$.  For any set or property $\mathcal S$, $\mathbf1_{\mathcal S}$ denotes its indicator function. If $h\in L^1(\mbbT)$, its Fourier coefficients are
\[
\widehat{h}(m)
  := \int_{\mathbb{T}} h(\theta)\,e(-m\theta)\,d\theta,
  \qquad m \in \mathbb{Z}.
\]
Finally, if $f,g\in L^1(\mathbb T)$, their Fourier convolution is defined by
\[
  (f*g)(\phi)
  =\int_{\mathbb T}f(\theta)\,g(\phi-\theta)\,d\theta.
\]
One immediately checks the standard identity
\begin{align}\label{fourier coefficient of convolution}
    \widehat{f*g}=\widehat f\cdot\widehat g.
\end{align}

\section{Exponential Sum Estimates}\label{exponential sum estimates}
In this section, we provide estimates to bound the exponential sums in our analysis. We first recall the well-known minor arc estimate for exponential sums over primes.
Let
\begin{align}\label{exponential sum over primes defn}
    S(\theta;N) = S(\theta) := \sum_{n\leq N}\Lambda(n)e(n\theta).
\end{align}
\begin{lemma}\label{minor arc estimate for Von Mangoldt}
Let $\beta\in\R$ be a real number for which there exist integers $a$ and $q$, with $(a,q)=1$, satisfying
\[
\left|\beta - \frac{a}{q}\right| \leq \frac{1}{q^2}.
\]
Then, for every $N\ge q$, we have
\[
\abs{S(\beta)} \ll (\log N)^{4}\left(\frac{N}{q^{1/2}} + N^{4/5} + (Nq)^{1/2}\right).
\]
\end{lemma}
\begin{proof}
    The result follows from \cite[Theorem 3.1]{vaughan} as the contribution from prime powers is negligible.
\end{proof}
In the next lemma, we show that among the phases $\set{\theta + m\alpha}_m$, only a limited number can be simultaneously approximated by rationals with small denominators.
\begin{lemma}\label{lemma for bound on exceptions that can be approximated using small denominator}
    Let $\alpha, \theta \in \mathbb{T}$. Suppose $r$ and $s$ are coprime integers satisfying
\begin{align}\label{rational appx to alpha}
    \left|\alpha - \frac{r}{s}\right| \le \frac{1}{s^2}.
\end{align}
Let $P, M$ and $N$ be positive real numbers such that $N\geq 16sP^2$. Let $\mcE(P;\theta)\subseteq[-M, M]\cap \Z$ denote the set of integers $m$ such that there exist coprime integers $a_m$ and $q_m$ with $q_m< P$ such that
\begin{align*}
    \abs{\theta+m\alpha-\frac{a_m}{q_m}}\leq \frac{P}{q_mN}.
\end{align*}
Then 
\begin{align*}
    \abs{\mcE(P;\theta)}\leq \frac{16MP^2}{s}+1.
\end{align*}
\end{lemma}
\begin{proof}
    Suppose for the sake of contradiction that 
    \begin{align*}
        \abs{\mcE(P;\theta)}> \frac{16MP^2}{s} +1.
    \end{align*}
    
    By the pigeon hole principle, we can find two integers $m_1, m_2\in \mcE(P;\theta)$ such that 
    \begin{align}\label{gap between m1 and m2}
        \abs{m_1-m_2}\leq \frac{s}{8P^2}.
    \end{align}
   By the definition of $\mcE(P;\theta)$, for each $j\in\{1,2\}$, there exist coprime integers $a_j$ and $q_j$ with $q_j<P$ such that
    \[
    \left|\theta+m_j\alpha-\frac{a_j}{q_j}\right|\le \frac{P}{Nq_j}.
    \]
    Applying the triangle inequality yields
    \[
    \left|(m_1-m_2)\alpha-\Bigl(\frac{a_1}{q_1}-\frac{a_2}{q_2}\Bigr)\right|\le \frac{P}{N}\Bigl(\frac{1}{q_1}+\frac{1}{q_2}\Bigr).
    \]
    A further application of the triangle inequality combined with~\eqref{rational appx to alpha} and~\eqref{gap between m1 and m2} gives
    \[
    \left|(m_1-m_2)\frac{r}{s}-\Bigl(\frac{a_1}{q_1}-\frac{a_2}{q_2}\Bigr)\right|\le \frac{P}{N}\Bigl(\frac{1}{q_1}+\frac{1}{q_2}\Bigr)+\frac{1}{8sP^2}.
    \]
    By~\eqref{gap between m1 and m2}, the left-hand side is nonzero. Hence,
    \[
    \frac{1}{sq_1q_2}\le \frac{P}{N}\Bigl(\frac{1}{q_1}+\frac{1}{q_2}\Bigr)+\frac{1}{8sP^2}.
    \]
    Multiplying both sides by $sq_1q_2$ and using $q_1+q_2\le 2P$ and $q_1q_2\le P^2$, we deduce that
    \[
    1\le \frac{2sP^2}{N}+\frac{1}{8}.
    \]
    As $N\geq 16sP^2$, the right-hand side is strictly less than $1$, a contradiction.
\end{proof}
Next, we show that, aside from a few exceptions, the exponential sums with phases $\{\theta + m\alpha\}_{m}$ are well controlled.
\begin{lemma}\label{bound for exponential sums with few exceptions}
    Let $\alpha, \theta, r, s, P, M, N$ and $\mcE(P;\theta)$ be as in Lemma~\ref{lemma for bound on exceptions that can be approximated using small denominator}. Let $a$ and $q$ be coprime integers with $2q<P$. For all integers $m\in [-M, M]\setminus \mcE(P;\theta)$, one has
    \[
        \left|\sum_{n\le N} e\Bigl(n\Bigl(\theta+m\alpha-\frac{a}{q}\Bigr)\Bigr)\right|\ll \frac{Nq}{P}.
    \]
\end{lemma}
\begin{proof}
 Fix $m\in [-M, M]\setminus \mcE(P;\theta)$. By the Dirichlet approximation theorem and the definition of $\mcE(P;\theta)$, there exists coprime integers $a_m$ and $q_m$ with $q_m\in [P, N/P]$ such that 
 \begin{align*}
     \abs{\theta+m\alpha-\frac{a_m}{q_m}}\leq \frac{P}{q_mN}\leq \frac{1}{q_m^2}.
 \end{align*}
 Then, by the triangle inequality,
\[
\left\|\theta+m\alpha-\frac{a}{q}\right\|\ge \left\|\frac{a_m}{q_m}-\frac{a}{q}\right\|-\frac{1}{q_m^2}.
\]
Since $2q<P\le q_m$, we have ${q_m}\neq{q}$ and thus
\[
\left\|\frac{a_m}{q_m}-\frac{a}{q}\right\|\ge \frac{1}{qq_m}.
\]
It follows that
\[
\left\|\theta+m\alpha-\frac{a}{q}\right\|\ge \frac{1}{qq_m}-\frac{1}{q_m^2}\gg \frac{P}{qN}.
\]
Finally, using the standard bound for exponential sums,
\[
\left|\sum_{n\le N} e\Bigl(n\Bigl(\theta+m\alpha-\frac{a}{q}\Bigr)\Bigr)\right|\ll \min\left\{N,\frac{1}{\left\|\theta+m\alpha-\frac{a}{q}\right\|}\right\},
\]
the result follows.
\end{proof}

\section{Weight Function}\label{weight function}
To introduce a smooth cut‐off in the Diophantine condition \eqref{defining condn}, we insert a bump function.  This modification will simplify the application of Fourier analysis on the torus $\mbbT$.

Fix a parameter $\delta \in \pare{0, \tfrac14}$, and let $w_\delta \colon \mbbT \to [0,1]$ be a $C^\infty(\mbbT)$ bump function supported on those $x$ with $\norm{x} \in [-2\delta, 2\delta]$, and satisfying
\[
w_\delta(x) = 1 \quad \text{for } \norm{x} \le \delta.
\]
The following lemma provides a convenient truncated Fourier expansion of the function $w_\delta$.
\begin{lemma}\label{Fourier expansion of the bump function}
For $\delta \in \pare{0, \tfrac14}$, let $w_\delta$ be as defined above. For $L,\,A \geq 1$, the function $w_\delta$ admits the Fourier expansion
\[
w_\delta\pare{x}
= \sum_{\abs{\ell} \le L} \widehat{w_\delta}(\ell)\, e(\ell x)
+ O_A\pare{\frac{1}{\pare{1 + L/\delta}^A}}.
\]
Furthermore, the Fourier coefficients satisfy the bound
\begin{align}\label{bound on Fourier coefficients}
    \widehat{w_\delta}(\ell) \ll \delta.
\end{align}
\end{lemma}
\begin{proof}
    Repeated applications of integration by parts yield
    \begin{align}\label{bound for tail fourier coeffs}
        \widehat{w_\delta}(\ell) = \int_\mbbT w_\delta(x)e(-\ell x)\,dx \ll_A \frac{1}{(1+\abs{\ell}/\delta)^{2A}}.
    \end{align}
    Using the estimate
    \[\sum_{\abs{\ell}>L}\frac{1}{(1+\abs{\ell}/\delta)^{2A}}\ll_A \frac{1}{\pare{1 + L/\delta}^A}\]
    the first claim follows.
    
    For the second claim, we observe that
    \[\abs{\widehat{w_\delta}(\ell)}\leq \int_{\mbbT}w_\delta(x)\,dx \leq \int_{-2\delta}^{2\delta}1\,dx\ll \delta.\]
\end{proof}
We will also need the following convolution estimate for the Fourier coefficients of $w_\delta$, which will be used to obtain a lower bound on the main term when counting three-term arithmetic progressions.
\begin{lemma}\label{convolution lower bound}
For $\delta\in\pare{0,\tfrac{1}{4}}$, let $w_\delta$ be as defined above. Then for any $L,\,A\ge1$, one has the lower bound
    \[\sum_{\abs{\ell}\leq L}\widehat{w_\delta}(\ell)^2\widehat{w_\delta}(-2\ell)\gg\delta^2+O_A\pare{\frac{1}{\pare{1+L/\delta}^A}}.\]
\end{lemma}
\begin{proof}
Set
\[v_\delta(\norm{x}) := \frac{w_\delta(-\norm{x}/2)}{2}.\]
Then,
\[\widehat{v_\delta}(m) = \frac{1}{2}\int_{-2\delta}^{2\delta} w_\delta\pare{\frac\theta2}e(-m\theta)\,d\theta = \int_{-\delta}^{\delta} w_\delta(\phi)e(-2m\phi)\,d\phi =  \widehat{w_\delta}(-2m).\]
Applying the convolution identity~\eqref{fourier coefficient of convolution} yields
\begin{align*}
    (w_\delta*w_\delta*v_\delta)(0) &= \sum_{\ell}\widehat{(w_\delta*w_\delta*v_\delta)}(\ell)e(0\cdot\ell)\\
    &=\sum_{\abs{\ell}\leq L}\widehat{w_\delta}(\ell)^2\widehat{w_\delta}(-2\ell)+\sum_{\abs{\ell}>L}\widehat{w_\delta}(\ell)^2\widehat{w_\delta}(-2\ell).
\end{align*}
The tail sum is controlled by the decay estimate~\eqref{bound for tail fourier coeffs},
\[\sum_{\abs{\ell}>L}\widehat{w_\delta}(\ell)^2\widehat{w_\delta}(-2\ell)\ll_A \frac{1}{\pare{1+L/\delta}^A}.\]

On the other hand, since $w_\delta\ge\mathbf{1}_{[-\delta,\delta]}$, using direct computation we have
\[(w_\delta*w_\delta*v)(0)\geq \mathbf{1}_{[-\delta, \delta]}*\mathbf{1}_{[-\delta, \delta]}*\frac{\mathbf{1}_{[-2\delta, 2\delta]}}{2}\geq \delta \mathbf{1}_{[-\delta, \delta]}*\frac{\mathbf{1}_{[-2\delta, 2\delta]}}{2}\gg \delta^2.\]
Combining these two estimates completes the proof.
\end{proof}
\section{Preliminaries for the Circle Method}\label{notation setup}
We set up the following notation to be used throughout the paper. Let $\alpha$ be irrational, and fix exponents
\[
0<\tau<\mu<\frac18.
\]
Let $J\ge1$ be the unique integer satisfying
\begin{align}\label{choice of J}
    4^J \;>\;\frac{1-5\mu}{1-8\mu}\;\ge\;4^{J-1},
\end{align}
and choose $\eta\in \pare{0, \tfrac1{20}}$ fixed such that
\begin{align}\label{choice of eta}
    \eta<\frac{4^J\,(1-8\mu)-(1-5\mu)}{16\,(4^J-1)}\leq\frac{1-8\mu}{4}.
\end{align}
Our choice of $J$ ensures that the expression in the middle above is positive. The reason for this peculiar choice of parameters $\eta$ and $J$ will be made clear below when we introduce a hierarchy of scales to aid our analysis.

Next, let $r$ and $s$ be coprime integers with $s$ large enough and such that~\eqref{rational appx to alpha} holds. Also let $N$ be such that  
\begin{align*}
    s = \floor{N^{1-4\mu-2\eta}}.
\end{align*}
It is useful to note that for $s$ large enough (or, interchangeably $N$ large enough) we have the bound $s\geq \sqrt{N}$ following from the second upper bound on $\eta$ in~\eqref{choice of eta}.
Set
\begin{align*}
    \delta := N^{-\tau},\quad  \text{ and }\quad M := N^{\mu}.
\end{align*}
For $\delta$ as defined above, set $w:=w_\delta$.

In our analysis we will frequently invoke Parseval’s identity to control second moments, and for that purpose we require an upper bound on the number of $n\le N$ satisfying $\norm{n\alpha}\leq \delta$. Using the exponential‐sum estimates established earlier, we now derive such a bound for the smoothed cutoff.
\begin{lemma}\label{bohr set is thin}
With notation as above, one has the bound
    \[\sum_{n\leq N}w(\norm{n\alpha}) \ll \delta N.\]
\end{lemma}
\begin{proof}
    By Lemma~\ref{Fourier expansion of the bump function} with cutoff $L=M$, for any $A\geq 1$ we have
    \[\sum_{n\leq N}w(\norm{n\alpha}) = \sum_{\abs{\ell}\leq M}\widehat{w}(\ell)\sum_{n\leq N}e(n\ell\alpha)+O_A\pare{\frac{1}{N^{A(\mu-\tau)}}}.\]
    As $\mu>\tau$, choosing $A$ large enough, we have 
    \[\sum_{n\leq N}w(\norm{n\alpha}) = \sum_{\abs{\ell}\leq M}\widehat{w}(\ell)\sum_{n\leq N}e(n\ell\alpha)+O\pare{N^{-50}}.\]
    Set
    \[
P:=\min\set{\sqrt{\frac{s}{64M}}, \sqrt{\frac{N}{16s}}},
\]
With $\alpha, r, s, P, M$ as above, let $\mcE\pare{P; 0}$ be as defined in Lemma~\ref{lemma for bound on exceptions that can be approximated using small denominator}. Using Lemma~\ref{lemma for bound on exceptions that can be approximated using small denominator} observe that $\abs{\mcE\pare{P; 0}}\leq 1$. It is clear that $0 \in \mcE(P; 0)$, and therefore
\[\mcE\pare{P; 0} = \set{0}.\]
Therefore, for each $\ell$ satisfying $1\le\abs{\ell}\le M$, applying Lemma~\ref{bound for exponential sums with few exceptions} with $a=q=1$, we obtain
\[
\sum_{n\le N}e(n\ell\alpha)\ll\frac{N}{P}.
\]
Hence, using~\eqref{bound on Fourier coefficients},
    \[\sum_{n\leq N}w(\norm{n\alpha})\ll \delta N+O\pare{\delta N\frac{M}{P}}. \]
On the other hand, using $\eta\in\pare{0, \tfrac{1}{20}}$ and $\mu<\tfrac18$, we have
\[
P\gg\min\set{N^{\tfrac12-\tfrac{5\mu+2\eta}{2}}, N^{2\mu+\eta}}
\gg \min\set{N^{\frac{9}{20}-\frac{5\mu}{2}}, N^{2\mu}} \gg N^\mu
=M.\]
Thus the result follows.
\end{proof}
To organize our circle method analysis, we introduce a hierarchy of scales
$$P_1>P_2>\dots>P_J,$$
where for each $1\le j\le J$
\begin{align}\label{defn of Pj}
    P_j = \frac{N^{p_j}}{256},
\quad
p_j = \frac{1}{4}-\frac{(2\cdot 4^j-5)}{3}\pare{\frac{1-8\mu}{4}-\eta}.
\end{align}
Note that, by the second upper bound on $\eta$ in~\eqref{choice of eta}, the exponents $p_j$ are strictly decreasing:
\[p_{j}-p_{j+1} = 2\cdot 4^j\pare{\frac{1-8\mu}{4}-\eta}>0.\]
Moreover, using the upper bound on $4^{J-1}$ from~\eqref{choice of J}, we obtain
\begin{equation}\label{lower bound pJ}
  \begin{aligned}
    p_J-\eta &= \frac{1}{4}-\frac{2\cdot 4^J-5}{3}\pare{\frac{1-8\mu}{4}-\eta}-\eta\\
    &= \frac{1}{4}-\frac{2\cdot 4^J-5}{3}\cdot \frac{1-8\mu}{4}+\eta\pare{\frac{2\cdot 4^J-5}{3}-1}\\
    &\geq \frac{1}{4}-\frac{8\cdot 4^{J-1}-5}{3}\cdot \frac{1-8\mu}{4}\\
    &\geq \frac{1}{4}-\frac{8(1-5\mu)-5(1-8\mu)}{12} =  0
\end{aligned}
\end{equation}
Thus, the scales $P_j$ are indeed monotonically decreasing, and each $P_j > 1$ for sufficiently large $N$ (in fact, $P_j\gg N^\eta$).

Our strategy will be to partition the set of phases $\set{\theta + m\alpha}_{\abs{m} \leq M}$ according to whether each phase admits a good rational approximation with denominator at most $P_j$. For each such partition, we then apply the appropriate minor arc bound given by Lemma~\ref{minor arc estimate for Von Mangoldt} to estimate its contribution accordingly. 

To control the size of each partition, we apply Lemma~\ref{lemma for bound on exceptions that can be approximated using small denominator}, which requires that $N\geq 16sP_j^2$ for all $j\leq J$. A direct computation shows that $p_1 = 2\mu + \eta$, so for sufficiently large $N$, remembering that $p_J\geq \eta$ from~\eqref{lower bound pJ}, we have
\begin{align}\label{bounds on the scales}
    \frac{N^{\eta}}{256}\leq P_j\leq \frac{1}{16}\sqrt{\frac{N}{s}}\leq N^{\frac14} \quad\text{ for }\quad 1\leq j\leq J. 
\end{align}

To ensure the contribution from each partition is acceptable, we will require the recurrence relation
\[
\frac{MP_j^2}{s} = o\pare{\sqrt{P_{j+1}}}
\]
to hold for all $j < J$, along with the condition that $M = o\pare{\sqrt{P_1}}$. A direct check shows that
\[p_{j+1}-4p_j = 10\mu+5\eta-2,\quad p_1 = 2\mu+\eta\]
and hence, for large enough $N$, the relations
\begin{align}\label{recursion works for the choice of P_j}
N^{\eta}\frac{M^2P_j^4}{s^2}\le P_{j+1}\quad(1\le j<J),
\qquad
\frac{M^2N^{\eta}}{256}\le P_1
\end{align}
hold, implying the above mentioned recurrence relations.

To ensure, among other things, that the modified major arcs in our setting below are disjoint, we will require that at most one element of the set $\set{\theta + m\alpha}_{\abs{m} \leq M}$ is well approximated by a rational with denominator at most $P_J$. By Lemma~\ref{lemma for bound on exceptions that can be approximated using small denominator}, this is guaranteed provided that
\[
P_J \leq \frac{1}{8} \sqrt{\frac{s}{M}}.
\]
To verify this, observe that by the upper bound on $\eta$ in~\eqref{choice of eta}, we have
\begin{align*}
    \frac{1-4\mu-2\eta-\mu}{2}-p_J&= \frac{1-5\mu}{2}-\eta - \frac{1}{4}+\frac{2\cdot 4^J-5}{3}\pare{\frac{1-8\mu}{4}-\eta}\\
    &=\frac{4^J(1-8\mu)-(1-5\mu)}{6}-\frac{2(4^J-1)}{3}\eta> 0.
\end{align*}
Hence, for sufficiently large $N$,
\begin{align}\label{upper bound on PJ}
     P_J\leq \frac{1}{16}\sqrt{\frac{s}{2M}}.
\end{align}

We now split the unit torus $\mbbT$ into \emph{major arcs} and \emph{minor arcs} as follows. For $q< P_J$, and $a$ coprime to $q$, and $\abs{m}\leq M$, define the \emph{major arc}
\[\mcM_{a, q, m} := \set{\theta\in \mbbT: \norm{\theta+m\alpha-\frac{a}{q}}\leq \frac{P_J}{qN}},\]
and set the \emph{minor arcs} to be
\[\mathfrak{m} := \mbbT\setminus \bigcup_{q< P_J}\bigcup_{\substack{a = 1\\ (a, q) = 1}}^q\bigcup_{\abs{m}\leq M}\mcM_{a, q, m}.\]

We now show that the major arcs $\set{\mcM_{a, q, m}}_{a, q, m}$ are pairwise disjoint.
\begin{lemma}\label{major arcs are disjoint}
Let $q_1,q_2< P_J$, and let $1\le a_1\le q_1$, $1\le a_2\le q_2$ satisfy $(a_1,q_1)=(a_2,q_2)=1$. Moreover, let $m_1,m_2\in\Z$ satisfy $\abs{m_1},\abs{m_2}\le M$. Then
\[
\mathcal{M}_{a_1,q_1,m_1}\cap\mathcal{M}_{a_2,q_2,m_2}\neq\emptyset
\]
if and only if 
\begin{align}\label{non emptyness condn for major arcs}
    a_1 = a_2, \quad q_1 = q_2, \quad\text{ and }\quad m_1= m_2 .
\end{align}
\end{lemma}
\begin{proof}
    Assume $\theta\in \mathcal{M}_{a_1,q_1,m_1}\cap\mathcal{M}_{a_2,q_2,m_2}$ and~\eqref{non emptyness condn for major arcs} does not hold. If $m_1 \neq m_2$, then the result immediately follows from the upper bounds on $P_J$ in~\eqref{bounds on the scales} and ~\eqref{upper bound on PJ}, and Lemma~\ref{lemma for bound on exceptions that can be approximated using small denominator}. Hence assume $m_1=m_2=m$ and that $a_1/q_1\neq a_2/q_2$ are distinct reduced fractions.
    By definition of the major arcs, we have
\[
\norm{\theta+m\alpha-\frac{a_1}{q_1}}\le \frac{P_J}{q_1N} \quad \text{and} \quad \norm{\theta+m\alpha-\frac{a_2}{q_2}}\le \frac{P_J}{q_2N}.
\]
By the triangle inequality,
\[
\frac{1}{q_1q_2}\leq \norm{\frac{a_1}{q_1}-\frac{a_2}{q_2}} \le \norm{\theta+m\alpha-\frac{a_1}{q_1}} + \norm{\theta+m\alpha-\frac{a_2}{q_2}} \le \frac{P_J}{q_1N} + \frac{P_J}{q_2N} \le \frac{2P_J}{N}.
\]
Since $q_1,q_2\le P_J$, this gives
\[
\frac{1}{P_J^2} \le \frac{2P_J}{N},
\]
contradicting~\eqref{bounds on the scales} and hence completing the proof.
\end{proof}
By Lemma~\ref{major arcs are disjoint}, the integral over $\mbbT$ decomposes as
\begin{align}\label{decomposition of the circle into major and minor arcs}
    \int_\mbbT = \sum_{\abs{m}\leq M}\sum_{q<Q_1}\sum_{\substack{a=1\\ (a, q)=1}}^q\int_{\mcM_{a, q, m}}+\int_{\mathfrak{m}}.
\end{align}
\section{Major and Minor Arc Estimates}\label{Major Arc and Minor Arc Estimates}
To implement the circle method in our setting, we require major and minor arc estimates for the following exponential sum
\begin{align}\label{exponential sum over bohr set primes}
    S_w(\theta) = S_{w}(\theta;N):= \sum_{n\le N} \Lambda(n)w\pare{\norm{\alpha n}}e(n\theta).
\end{align}
Recall the definition of $S(\theta)$ from~\eqref{exponential sum over primes defn}. The following lemma provides a major arc estimate for the exponential sum in~\eqref{exponential sum over bohr set primes}.
\begin{lemma}\label{major arc estimate} Let $q<P_J$ and let $1\leq a\leq q$ be such that $(a, q) = 1$. Suppose $m\in [-M, M]\cap\Z$, and let $\theta\in \mcM_{a, q, m}$. Then
    \[S_w(\theta) = \widehat{w}(m)S(\theta+m\alpha)+O\pare{\delta N^{1-\frac{\eta}{3}}}\]
and
\[S_w(-2\theta) = \widehat{w}(-2m)S(-2(\theta+m\alpha))+O\pare{\delta N^{1-\frac{\eta}{3}}}.
\]
\end{lemma}
\begin{proof}
    Apply Lemma \ref{Fourier expansion of the bump function} with $L = 2M$ and $x = n\alpha$ to obtain
    \[S_w(\theta) = \sum_{\abs{\ell}\leq 2M}\widehat{w}(\ell)S(\theta+\ell\alpha)+O\pare{N^{-50}}.\]
    Let $\mcE_0:=[-2M, 2M]\cap \Z$. For $1\le j\le J$, define $\mcE_j\subseteq\mcE_{j-1}$ as the set of integers $\ell\in \mcE_{j-1}$ such that there exists coprime integers $a,\,q$ with $q< P_j$ satisfying
    \begin{align*}
        \abs{\theta+\ell\alpha-\frac{a}{q}}\leq \frac{P_j}{qN}.
    \end{align*}
By Dirichlet’s approximation theorem, for each $\ell \in \mathcal{E}_{j-1} \setminus \mathcal{E}_j$, there exist coprime integers $a$, $q$ with $q \in [P_j, N/P_j]$ such that
 \begin{align*}
     \abs{\theta+\ell\alpha-\frac{a}{q}}\leq \frac{P_j}{qN}\leq \frac{1}{q^2}.
 \end{align*}    
 Applying Lemma~\ref{minor arc estimate for Von Mangoldt}, for every $\ell\in \mcE_{j-1}\setminus \mcE_j$ we have
 \begin{align*}
     \abs{S(\theta+\ell\alpha)}\ll N(\log N)^{4}\pare{\frac{1}{\sqrt{P_j}}+\frac{1}{N^{1/5}}}\ll \frac{N(\log N)^{4}}{\sqrt{P_j}}
 \end{align*}
From Lemma~\ref{lemma for bound on exceptions that can be approximated using small denominator} and the upper bounds on $P_J$ in~\eqref{bounds on the scales} and~\eqref{upper bound on PJ}, we have $\abs{\mcE_J}\leq 1$. Since $\theta\in \mcM_{a, q, m}$, it follows that if $\mcE_J$ is non-empty, then $m\in \mcE_J$. 

Using the bound in~\eqref{bound on Fourier coefficients}, we conclude
    \begin{align*}
        S_w(\theta) = \widehat{w}(m)S(\theta+m\alpha)+O\pare{\delta N(\log N)^{4}\sum_{0\leq j< J}\frac{\abs{\mcE_j}}{\sqrt{P_{j+1}}}}+O\pare{N^{-50}}.
    \end{align*}
    Finally, invoking Lemma~\ref{lemma for bound on exceptions that can be approximated using small denominator} to control $\abs{\mcE_j}$, and using~\eqref{recursion works for the choice of P_j}, we obtain
    \begin{align*}
        S_w(\theta) &= \widehat{w}(m)S(\theta+m\alpha)+O\pare{\delta N(\log N)^4 JN^{-\frac{\eta}{2}}}\\
        &= \widehat{w}(m)S(\theta+m\alpha)+O\pare{\delta N^{1-\frac{\eta}{3}}}
    \end{align*}
    as required. 
    
    For $\theta\in \mcM_{a, q, m}$, observe that if $ q $ is even, the $-2\theta\in\mcM_{q-a, \frac{q}{2}, -2m}$. Hence, the estimate for $ S_w(-2\theta) $ follows from estimate for $S_w(\theta)$. 
    
    Now suppose $q$ is odd.  For this case, define, for any $ d<2P_J$, $1\leq c\leq  d$ with $(c,  d) = 1$, and $\abs{\ell}\leq M$,
    \begin{align*}
        \mcM^\prime_{c,  d, \ell}:=\set{\theta\in \mbbT\colon\norm{\theta+\ell\alpha-\frac{c}{ d}}\leq \frac{2P_J}{ d N}}.
    \end{align*}
    Then, for $\theta\in\mcM_{a, q, m}$, there exists an integer $1\leq b\leq q$ with $(b, q) = 1$, such that $-2\theta\in\mcM^\prime_{b, q, -2m}$. Now, by replacing the parameters $P_1,\, \dots,\, P_J $ with $2P_1,\,\dots,\, 2P_J$, the desired estimate for $ S_w(-2\theta) $ follows by analogous argument as that used for $S_w(\theta)$.
\end{proof}
The next lemma provides a point-wise bound for the exponential sum $S_w(\theta)$ on the minor arcs.
\begin{lemma}\label{minor arc estimate}
    We have the following uniform bound:
     \[\sup_{\theta\in \mathfrak{m}}\abs{S_w(\theta)}\ll \delta N^{1-\frac{\eta}{3}}.\]
\end{lemma}
\begin{proof}
The argument mirrors the proof of Lemma~\ref{major arc estimate}, except that for $\theta\in \mathfrak{m}$, $\mcE_J$ must be empty. Hence,
\begin{align*}
        S_w(\theta) = O\pare{\delta N(\log N)^{4}\sum_{0\leq j< J}\frac{\abs{\mcE_j}}{\sqrt{P_{j+1}}}}+O\pare{N^{-50}} = O\pare{\delta N^{1-\frac{\eta}{3}}}.
    \end{align*}
\end{proof}
\section{Proof of Theorem \ref{theorem 3APs}}\label{proof of theorem 3APs}
With the major and minor arc estimates established, we now proceed to complete the proof of Theorem~\ref{theorem 3APs}. By the orthogonality relation
\begin{align}\label{orthogonality relation}
    \int_{\mbbT}e((n-m)\theta)\,d\theta=\mathbf{1}_{n=m},
\end{align}
where $m,\, n\in \Z$, it suffices to establish the lower bound
\begin{align}\label{required lower bound for 3aps}
    \int_{\mbbT}S_w(\theta)^2\,S_w(-2\theta)\,d\theta
\gg \delta^2N^2.
\end{align}
Applying~\eqref{decomposition of the circle into major and minor arcs}, we have the decomposition
\begin{align}\label{circle method decomposition}
    \int_{\mbbT} S_w(\theta)^2\,S_w(-2\theta)\,d\theta = R_{\mcM} + R_{\mathfrak{m}},
\end{align}
where
\begin{align}\label{major arc contribution rep}
    R_{\mcM} := \sum_{q< Q_1} \sum_{\substack{a = 1\\(a,q)=1}}^q \sum_{|m|\le M} \int_{\mcM_{a,q,m}} S_w(\theta)^2\,S_w(-2\theta)\,d\theta,
\end{align}
and
\begin{align}\label{minor arc contribution rep}
    R_{\mathfrak{m}} := \int_{\mathfrak{m}} S_w(\theta)^2\,S_w(-2\theta)\,d\theta.
\end{align}
We will show that the major‐arc contribution $R_\mcM$ dominates and that the minor-arc contribution is negligible, thereby proving the desired bound.
\begin{lemma}\label{bound on minor arc}
We have
    \[R_\mathfrak{m}\ll \delta^2 N^{2-\frac\eta4}.\]
\end{lemma}
\begin{proof}
From~\eqref{minor arc contribution rep},
    \[R_\mathfrak{m}\ll \sup_{\theta\in\mathfrak{m}}\abs{S_w(\theta)}\int_{\mbbT}\abs{S_w(\theta)S_w(-2\theta)}\,d\theta.\]
An application of the Cauchy–Schwarz inequality gives
    \[R_\mathfrak{m}\ll \sup_{\theta\in\mathfrak{m}}\abs{S_w(\theta)}\pare{\int_{\mbbT}\abs{S_w(\theta)}^2\,d\theta}^{1/2}\pare{\int_{\mbbT}\abs{S_w(-2\theta)}^2\,d\theta}^{1/2}.\]
Invoking Parseval’s identity together with Lemma~\ref{bohr set is thin} shows that
\[\int_\mbbT \abs{S_w(\theta)}^2\,d\theta, \quad\int_\mbbT \abs{S_w(-2\theta)}^2\,d\theta\ll \delta N(\log N)^2.\]
Hence
    \[R_\mathfrak{m}\ll \sup_{\theta\in\mathfrak{m}}\abs{S_w(\theta)}\delta N(\log N)^2,\]
    and the claimed bound follows at once from Lemma~\ref{minor arc estimate}.
\end{proof}
\begin{lemma}\label{bound on major arc}
We have
    \[R_\mcM \gg \delta^2N^2.\]
\end{lemma}
\begin{proof}
    Invoking Lemma~\ref{major arc estimate}, we see that for every integer $q<P_J$, each $a$ with $1\le a\le q$ and $(a,q)=1$, and every $m\in[-M,M]\cap\mathbb{Z}$, 
    \begin{align}\label{one application of major arc estimate}
        \int_{\mcM_{a, q, m}}S_w(\theta)^2S_w(-2\theta)\,d\theta = \widehat{w}(m)&\int_{\mcM_{a, q, m}}S(\theta+m\alpha)S_w(\theta)S_w(-2\theta)\,d\theta\\\
        &+O\pare{ \delta N^{1-\frac{\eta}{3}}\int_{\mcM_{a, q, m}}\abs{S_w(\theta)S_w(-2\theta)}\,d\theta}\nonumber.
    \end{align}
    Since the arcs $\set{\mcM_{a,q,m}}$ are disjoint (Lemma~\ref{major arcs are disjoint}), summing over $m, \,q$ and $a$ yields
    \[\sum_{\abs{m}\leq M}\sum_{q< Q_1}\sum_{\substack{a=1\\ (a, q)=1}}^q\int_{\mcM_{a, q, m}}\abs{S_w(\theta)S_w(-2\theta)}\,d\theta\ll\int_{\mbbT}\abs{S_w(\theta)S_w(-2\theta)}\,d\theta.\]
    By the Cauchy-Schwarz inequality, Parseval’s identity and Lemma~\ref{bohr set is thin},
    \begin{align}\label{first application of parseval and cauchy}
        \sum_{\abs{m}\leq M}\sum_{q< Q_1}\sum_{\substack{a=1\\ (a, q)=1}}^q\int_{\mcM_{a, q, m}}\abs{S_w(\theta)S_w(-2\theta)}\,d\theta\ll \delta N(\log N)^2.
    \end{align}
    Next, a second application of Lemma~\ref{major arc estimate} gives
    \begin{align}\label{two applications of major arc estimate}
        \int_{\mcM_{a, q, m}}S(\theta+m\alpha)S_w(\theta)S_w(-2\theta)\,d\theta = \widehat{w}(m)&\int_{\mcM_{a, q, m}}S(\theta+m\alpha)^2S_w(-2\theta)\,d\theta\\
        +&O\pare{\delta N^{1-\frac{\eta}{3}}\int_{\mcM_{a, q, m}}\abs{S(\theta+m\alpha)S_w(-2\theta)}\,d\theta}\nonumber.
    \end{align}
Using the same strategy as before, one shows that
\begin{align}\label{second application of parseval and cauchy}
     \sum_{\abs{m}\leq M}\sum_{q< Q_1}\sum_{\substack{a=1\\ (a, q)=1}}^q\int_{\mcM_{a, q, m}}\abs{S(\theta+m\alpha)S_w(-2\theta)}\,d\theta\ll \delta^{1/2}N(\log N)^2.
\end{align}
A final application of Lemma~\ref{major arc estimate} then gives, for each major arc $\mcM_{a,q,m}$,
\begin{align}\label{three application of major arc estimate}
    \int_{\mcM_{a, q, m}}S(\theta+m\alpha)^2S_w(-2\theta)\,d\theta = \widehat{w}(-2m)&\int_{\mcM_{a, q, m}}S(\theta+m\alpha)^2S(-2(\theta+m\alpha))\,d\theta\\
    +&O\pare{\delta N^{1-\frac{\eta}{3}}\int_{\mcM_{a, q, m}}\abs{S(\theta+m\alpha)}^2\,d\theta}.\nonumber
\end{align}
By disjointness of the major arcs (Lemma~\ref{major arcs are disjoint}), together with Parseval’s identity, one obtains
\begin{align}\label{third application of parseval and cauchy}
     \sum_{\abs{m}\leq M}\sum_{q< Q_1}\sum_{\substack{a=1\\ (a, q)=1}}^q\int_{\mcM_{a, q, m}}\abs{S(\theta+m\alpha)}^2\,d\theta\ll N(\log N)^2.
\end{align}
Combining \eqref{major arc contribution rep}, \eqref{one application of major arc estimate},
\eqref{first application of parseval and cauchy}, \eqref{two applications of major arc estimate},
\eqref{second application of parseval and cauchy}, \eqref{three application of major arc estimate},
and \eqref{third application of parseval and cauchy}, we arrive at
    \begin{align}\label{final after using major arc estimates}
        R_\mcM = \sum_{\abs{m}\leq M}\widehat{w}(m)^2\widehat{w}(-2m)\sum_{q< Q_1}\sum_{\substack{a=1\\ (a, q)=1}}^q\int_{\mcM_{a, q, 0}}S(\theta)^2S(-2\theta)\,d\theta+O(\delta^2N^{2-\frac{\eta}{4}}).
    \end{align}
Set 
\begin{align*}
    \mcM^{(0)} := \bigcup_{q<Q_1}\bigcup_{\substack{a = 1\\ (a, q) = 1}}^q\mcM_{a, q, 0}.
\end{align*}
Then
    \[\sum_{q< Q_1}\sum_{\substack{a=1\\ (a, q)=1}}^q\int_{\mcM_{a, q, 0}}S(\theta)^2S(-2\theta)\,d\theta = \int_{\mbbT}S(\theta)^2S(-2\theta)\,d\theta-\int_{\mbbT\setminus \mcM^{(0)}}S(\theta)^2S(-2\theta)\,d\theta.\]
By Lemma~\ref{minor arc estimate for Von Mangoldt} and the lower bound on $P_J$ from \eqref{bounds on the scales},
    \begin{align}\label{upper bound for S(theta) in zeroth complement}
        \sup_{\theta\in \mbbT\setminus \mcM^{(0)}}\abs{S(\theta)}\ll \frac{N(\log N)^{4}}{\sqrt{P_J}}\ll N^{1-\frac{\eta}{3}}.
    \end{align}
Another application of the Cauchy-Schwarz inequality and Parseval's identity then gives
    \[\int_{\mbbT\setminus \mcM^{(0)}}\abs{S(\theta)^2S(-2\theta)\,d\theta}\ll \sup_{\theta\in \mbbT\setminus \mcM^{(0)}}\abs{S(\theta)}\int_\mbbT \abs{S(\theta)S(-2\theta)}\,d\theta\ll N^{2-\frac{\eta}{4}}.\]
It is well‐known (cf.\ \cite{fan}) that
    \[\int_{\mbbT}S(\theta)^2S(-2\theta)\,d\theta = \sum_{\substack{n_1, n_2, n_3\leq N\\n_1+n_3 = 2n_2}}\Lambda(n_1)\Lambda(n_2)\Lambda(n_3)\gg N^2.\]
Hence
    \begin{align*}
         \sum_{q< Q_1}\sum_{\substack{a=1\\ (a, q)=1}}^q\int_{\mcM_{a, q, 0}}S(\theta)^2S(-2\theta)\,d\theta \gg N^2.
    \end{align*}
 Substituting this into \eqref{final after using major arc estimates} and applying Lemma~\ref{convolution lower bound} completes the proof.
\end{proof}
Substituting the bounds from Lemma~\ref{bound on minor arc} and Lemma~\ref{bound on major arc} into \eqref{circle method decomposition}, we obtain the desired lower bound in \eqref{required lower bound for 3aps}, thereby completing the proof of Theorem~\ref{theorem 3APs}.
\section{Proof of Theorem \ref{theorem binary goldbach}}\label{proof of theorem binary goldbach}
Finally, we turn to the proof of Theorem~\ref{theorem binary goldbach}.  To obtain stronger cancellation in $S(\theta)$, we impose an additional constraint on $\mu$.  Define a sequence $\{\mu_t^*\}_{t\ge1}$ increasing to $\tfrac18$ by solving
$$4^{t} = \frac{1-5\mu^*_t}{1-8\mu^*_t}.$$
That is, 
$$\mu_t^* = \frac{4^t-1}{8\cdot 4^t-5}.$$
Recall the definition of $P_J$ from~\eqref{defn of Pj}. We would like to give a better lower bound on $P_J$ in order to facilitate cancellations in $S(\theta)$. Since $\mu_t^*\nearrow \tfrac18$ as $t\to\infty$, we may choose $m$ so that
$\tau<\mu_m^*<\tfrac18$.  Then pick $\mu\in(\tau,\mu_m^*)$ sufficiently close to $\mu_m^*$ so that $J=m$ and
\[
\frac{2(1-5\mu)}{3}-\frac{4^J(1-8\mu)}{6}
>
\frac{2(1-5\mu)}{3}
-\frac{1-8\mu}{1-8\mu_m^*}\,\frac{1-5\mu_m^*}{6}
>
\frac{1-5\mu}{3}
>
\mu.
\]
It follows that
\begin{equation}\label{lower bound on J in bin goldbach}
  P_J \gg N^{\mu+\eta}.
\end{equation}

Using the orthogonality relation \eqref{orthogonality relation}, it suffices to show
\[
\int_{\mathbb T} S(\theta)\,S_w(\theta)\,e(-n\theta)\,d\theta
\gg
\delta\,\mathfrak S(n)\,n
\]
for almost all even $n$, where
\[\mathfrak{S}(n):=2\prod_{\substack{p\mid n\\p>2}}\pare{\frac{p-1}{p-2}}\prod_{p>2}\pare{1-\frac{1}{(p-1)^2}}.\]
Since
\begin{align*}
    \int_{\mbbT}S(\theta)^2e(-n\theta)\,d\theta\gg \sum_{\substack{m_1, m_2\\m_1+m_2 = n}}\Lambda(m_1)\Lambda(m_2)\gg \mathfrak{S}(n)n
\end{align*}
for almost all even $n$ (cf.\ \cite[\S3.2]{vaughan}), Theorem \ref{theorem binary goldbach} will follow from the next lemma.

\begin{lemma}
We have
\begin{align*}
    \sum_{n\leq N}\abs{\int_\mbbT S(\theta)(S_w(\theta)-\widehat{w}(0)S(\theta) )e(-n\theta)\,d\theta}^2\ll \delta^2N^{3-\frac\eta4}.
\end{align*}
\end{lemma}
\begin{proof}
    By Bessel's inequality,
    \begin{align*}
    \sum_{n\leq N}\abs{\int_\mbbT S(\theta)(S_w(\theta)-\widehat{w}(0)S(\theta) )e(-n\theta)\,d\theta}^2\leq \int_\mbbT \abs{S(\theta)(S_w(\theta)-\widehat{w}(0)S(\theta) )}^2\,d\theta.
\end{align*}
Let 
\begin{align*}
    \mcM^{(0)} := \bigcup_{q<Q_1}\bigcup_{\substack{a = 1\\ (a, q) = 1}}^q\mcM_{a, q, 0},
\end{align*}
and split the integral as
\begin{align*}
    \int_\mbbT = \int_{\mcM^{(0)}}+\int_{\mbbT\setminus \mcM^{(0)}}.
\end{align*}
On $\mathcal M^{(0)}$, Lemma~\ref{major arc estimate} and Parseval's identity gives
\begin{align*}
    \int_{\mcM^{(0)}} \abs{S(\theta)(S_w(\theta)-\widehat{w}(0)S(\theta))}^2\,d\theta &\leq \sup_{\theta\in \mcM^{(0)}}\abs{S_w(\theta)-\widehat{w}(0)S(\theta)}^2\int_{\mbbT}\abs{S(\theta)}^2\,d\theta\\
    &\ll \delta^2 N^{3-\frac\eta4}.
\end{align*}
Moreover, on $\mathbb T\setminus\mathcal M^{(0)}$, Parseval's identity, Lemma~\ref{bohr set is thin} and~\eqref{bound on Fourier coefficients} gives
\begin{align*}
    \int_{\mbbT\setminus\mcM^{(0)}} \abs{S(\theta)(S_w(\theta)-\widehat{w}(0)S(\theta))}^2\,d\theta &\leq \sup_{\theta\in \mbbT\setminus\mcM^{(0)}}\abs{S(\theta)}^2\int_{\mbbT}\abs{S_w(\theta)-\widehat{w}(0)S(\theta)}^2\,d\theta\\
    &\ll \sup_{\theta\in \mbbT\setminus\mcM^{(0)}}\abs{S(\theta)}^2\delta N.
\end{align*}
Using~\eqref{lower bound on J in bin goldbach} and Lemma~\ref{minor arc estimate for Von Mangoldt}, we see that
\[\sup_{\theta\in \mbbT\setminus\mcM^{(0)}}\abs{S(\theta)}\ll N^{1-\frac{(\mu+\eta)}{2}}.\]
Therefore
\begin{align*}
    \int_{\mbbT\setminus\mcM^{(0)}} \abs{S(\theta)(S_w(\theta)-\widehat{w}(0)S(\theta))}^2\,d\theta &\ll \delta N^{3-\mu-\frac{\eta}{3}}\\
    &\ll \delta^2 N^{3-\frac\eta4}.
\end{align*}
This completes the proof.
\end{proof}


\end{document}